\documentclass[11pt]{amsart}
\usepackage{amssymb,amsthm,amsmath,epsfig,latexsym}
\usepackage{calc,times,verbatim,graphicx}
\usepackage{epstopdf}
\usepackage{geometry} 
\begin{document}

\newcommand{\mmbox}[1]{\mbox{${#1}$}}
\newcommand{\proj}[1]{\mmbox{{\mathbb P}^{#1}}}
\newcommand{\Cr}{C^r(\Delta)}
\newcommand{\CR}{C^r(\hat\Delta)}
\newcommand{\affine}[1]{\mmbox{{\mathbb A}^{#1}}}
\newcommand{\Ann}[1]{\mmbox{{\rm Ann}({#1})}}
\newcommand{\caps}[3]{\mmbox{{#1}_{#2} \cap \ldots \cap {#1}_{#3}}}
\newcommand{\Proj}{{\mathbb P}}
\newcommand{\N}{{\mathbb N}}
\newcommand{\Z}{{\mathbb Z}}
\newcommand{\C}{{\mathbb C}}
\newcommand{\R}{{\mathbb R}}
\newcommand{\A}{{\mathcal{A}}}
\newcommand{\Tor}{\mathop{\rm Tor}\nolimits}
\newcommand{\Ext}{\mathop{\rm Ext}\nolimits}
\newcommand{\Hom}{\mathop{\rm Hom}\nolimits}
\newcommand{\im}{\mathop{\rm Im}\nolimits}
\newcommand{\rank}{\mathop{\rm rank}\nolimits}
\newcommand{\supp}{\mathop{\rm supp}\nolimits}
\newcommand{\arrow}[1]{\stackrel{#1}{\longrightarrow}}
\newcommand{\CB}{Cayley-Bacharach}
\newcommand{\coker}{\mathop{\rm coker}\nolimits}
\sloppy
\newtheorem{defn0}{Definition}[section]
\newtheorem{prop0}[defn0]{Proposition}
\newtheorem{quest0}[defn0]{Question}
\newtheorem{thm0}[defn0]{Theorem}
\newtheorem{lem0}[defn0]{Lemma}
\newtheorem{corollary0}[defn0]{Corollary}
\newtheorem{example0}[defn0]{Example}
\newtheorem{remark0}[defn0]{Remark}
\newtheorem{conj0}[defn0]{Conjecture}

\newenvironment{defn}{\begin{defn0}}{\end{defn0}}
\newenvironment{prop}{\begin{prop0}}{\end{prop0}}
\newenvironment{quest}{\begin{quest0}}{\end{quest0}}
\newenvironment{thm}{\begin{thm0}}{\end{thm0}}
\newenvironment{lem}{\begin{lem0}}{\end{lem0}}
\newenvironment{cor}{\begin{corollary0}}{\end{corollary0}}
\newenvironment{exm}{\begin{example0}\rm}{\end{example0}}
\newenvironment{rem}{\begin{remark0}\rm}{\end{remark0}}
\newenvironment{conj}{\begin{conj0}}{\end{conj0}}

\newcommand{\defref}[1]{Definition~\ref{#1}}
\newcommand{\propref}[1]{Proposition~\ref{#1}}
\newcommand{\thmref}[1]{Theorem~\ref{#1}}
\newcommand{\lemref}[1]{Lemma~\ref{#1}}
\newcommand{\corref}[1]{Corollary~\ref{#1}}
\newcommand{\exref}[1]{Example~\ref{#1}}
\newcommand{\secref}[1]{Section~\ref{#1}}
\newcommand{\remref}[1]{Remark~\ref{#1}}
\newcommand{\conjref}[1]{Conjecture~\ref{#1}}
\newcommand{\questref}[1]{Question~\ref{#1}}

\newcommand{\std}{Gr\"{o}bner}
\newcommand{\jq}{J_{Q}}


\title{On the geometry of real or complex supersolvable line arrangements}

\author{Benjamin Anzis and \c{S}tefan O. Toh\v{a}neanu}

 \keywords{Dirac-Motzkin conjecture, Slope problem, supersolvable arrangements.\\
\indent Authors' addresses: Department of Mathematics, University of Idaho, Moscow, ID 83844, anzi4123@vandals.uidaho.edu, tohaneanu@uidaho.edu.}

\begin{abstract}
\noindent Given a rank 3 real arrangement $\mathcal A$ of $n$ lines in the projective plane, the Dirac-Motzkin conjecture (proved by Green and Tao in 2013) states that for $n$ sufficiently large, the number of simple intersection points of $\mathcal A$ is greater than or equal to $n/2$. With a much simpler proof we show that if $\mathcal A$ is supersolvable, then the conjecture is true for any $n$ (a small improvement of original conjecture). The Slope problem (proved by Ungar in 1982) states that $n$ non-collinear points in the real plane determine at least $n-1$ slopes; we show that this is equivalent to providing a lower bound on the multiplicity of a modular point in any (real) supersolvable arrangement. In the second part we find connections between the number of simple points of a supersolvable line arrangement, over any field of characteristic 0, and the degree of the reduced Jacobian scheme of the arrangement. Over the complex numbers even though the Sylvester-Gallai theorem fails to be true, we conjecture that the supersolvable version of the Dirac-Motzkin conjecture is true.
\end{abstract}
\maketitle

\section{Introduction}

Let $\A$ be a line arrangement in $\mathbb P^2$. Suppose $\ell_1,\ldots,\ell_n\in\mathbb K[x,y,z]$ are the defining equations of the lines of $\A$, and assume $\dim_{\mathbb K}Span(\ell_1,\ldots,\ell_n)=3$ (i.e.\ that $\A$ has full rank, equal to 3).

An intersection point $P$ of any two of the lines of $\A$ is a singularity, denoted by $P \in Sing(\A)$. The number of lines of $\A$ that intersect at a singular point $P \in Sing(\A)$ is called the {\em multiplicity} of $P$, denoted $m(P,\A)$, or just $m_P$ if the line arrangement is clear. The {\em multiplicity of $\A$}, denoted $m(\A)$, is $m(\A)=\max\{m_P \mid P\in Sing(\A)\}$. {\em Simple points} are singular points of multiplicity 2, and the set of such points in $\A$ will be denoted $Sing_2(\A)$.

In general, a hyperplane arrangement is {\em supersolvable} if its intersection lattice has a maximal chain of modular elements (\cite{St}). For the case of line arrangements $\A\subset \mathbb P^2$, supersolvability is equivalent to the existence of a $P\in Sing(\A)$ such that for any other $Q\in Sing(\A)$, the line connecting $P$ and $Q$ belongs to $\A$ (in other words, there is an intersection point $P$ that ``sees" all the other intersection points through lines of $\A$). Such a point $P$ will be called {\em modular}. Supersolvable hyperplane arrangements are an important class of hyperplane arrangements, capturing a great deal of topological (over $\mathbb C$ they are also known as fiber-type arrangements, see \cite{JaTe}), combinatorial, and homological information (see \cite{OrTe} for lots of information). Due to their highly combinatorial content, they are the best understood arrangements, yet there are some questions and problems like the ones we investigate here, which we feel deserve to be analyzed.

\medskip

The goals of these notes are to understand in a set-theoretical and combinatorial context the singularities of various types of supersolvable line arrangements; e.g. how many simple singularities such a divisor can have (and in general, how many singularities are there all together), or what is the multiplicity of a modular singularity. The driving forces behind our results are the following two classical problems about configuration of points in the real plane and the lines they determine. Since any two points determine a line, under the classical duality (points $\leftrightarrow$ lines) any two lines intersect at a point, so we are translating these problems into questions about line arrangements in the projective plane.

The {\em Dirac-Motzkin Conjecture} states that if $\A$ is a full rank real arrangement of $n$ lines in $\mathbb P^2$, then, for $n$ sufficiently large, $\lvert Sing_2(\A) \rvert \geq n/2$. This conjecture has been proven only recently by Ben Green and Terence Tao in \cite{GrTa}, where they also solve another famous related problem (the Orchard Problem). They present the unique class of examples, up to projective transformations, (called ``B\"{o}r\"{o}czky examples'', see \cite[Proposition 2.1 (i)]{GrTa}) for which the bound is attained (of course $n$ must be even). On a side note, there are only two known examples for which $\lvert Sing_2(\A) \rvert < n/2$ (see the brief history in Section 2); based on this, Gr\"{u}nbaum conjectured that if $n\neq 7, 13$, then $\lvert Sing_2(\A) \rvert \geq n/2$ (see \cite{Gr}).

$\bullet$ With a very short proof we show (see Theorem \ref{Main1}) that if $\A \subset \Proj^2$ is a full rank real supersolvable line arrangement, then $\lvert Sing_2(\A) \rvert \geq \frac{\lvert \A \rvert}{2}$. Though it is not stated in \cite{GrTa}, by a simple calculation one should remark that the B\"{o}r\"{o}czky examples are supersolvable line arrangements. We note here that our proof may give insights towards proving Dirac-Motzkin conjecture for complex supersolvable line arrangements (see the discussions in Section 3.2).

\medskip

In \cite{Sc}, Scott proposed the following {\em Slope Problem}: Given $n\geq 3$ points in the (real) plane, not all collinear, they determine at least $n-1$ lines of distinct slopes. The first to prove this conjecture was Peter Ungar in \cite{Un}, using a beautiful yet difficult argument.

$\bullet$ In Proposition \ref{equivalence} we show that the Slope problem is equivalent to showing that if $\A$ is any full rank real supersolvable line arrangement in $\mathbb P^2$, then $m(\A)\geq(|\A|-1)/2$. So providing an independent proof (that we don't have at the moment) of the statement about the supersolvable line arrangements will lead to an alternative proof of the slope problem.

\medskip

At the beginning of Section 3 the results presented are valid over any field of characteristic 0. We prove (Proposition \ref{bound}) a lower bound for $\lvert Sing_2(\A) \rvert$, when $\A$ is a supersolvable line arrangement, that involves $\lvert Sing(\A)\rvert$, and we remark that the bound is attained also by the B\"{o}r\"{o}czky examples. Next we give an example which shows the upper bound for the degree of the reduced Jacobian scheme of a supersolvable line arrangement obtained in \cite[Proposition 3.1]{To2} is sharp.

The remaining part of Section 3 is dedicated to analyzing the Dirac-Motzkin conjecture for complex supersolvable arrangements (see Conjecture \ref{conjecture}; basically the first bullet above with ``real'' replaced by ``complex'').

\medskip

Lots of our arguments are based on the following simple observation, which is derived from the fact that any two lines in a projective plane must intersect: Let $\A\subset\mathbb P^2$ be any arrangement of $n$ lines, over any field. Let $\ell\in\A$ be any line, and suppose it has exactly $s$ intersection points: $P_1,\ldots, P_s$. Then $$\sum_{i=1}^s(m_{P_i}-1)=n-1.$$

\section{Supersolvable line arrangements over the real numbers}

Let $\A$ be a full rank supersolvable line arrangement in $\mathbb P^2$. Suppose $\lvert \A \rvert=n$, and denote the multiplicity of $\A$ by $m:=m(\A)$. We begin by reproving \cite[Lemma 2.1]{To1}, which is a statement that is true when working over any field.

\begin{lem}\label{MaxModular}
Let $\A$ be a supersolvable line arrangement with a modular point $P$. Let $Q \in Sing(\A)$ not be modular. Then $$m(P,\A) > m(Q,\A).$$ Therefore, any intersection point of maximum multiplicity $m$ is modular.
\end{lem}

\begin{proof}
Let $s = m(Q,\A)$ and let the lines passing through $Q$ be $\{\ell_1, \ldots, \ell_s\}\subset \A$. Since $Q$ is not modular, there exists a point $P'\in Sing(\A)$ with $P'\notin \ell_i,i=1,\ldots,s$. Because $P$ is modular, there is a line $\ell_{PP'}\in\A$ through $P$ and $P'$. Since $P'\in Sing(\A)$ there is another line $\ell\in\A$ through $P'$ and not passing through $P$. $\ell$ intersects the $\ell_i$'s in $s$ points, say $\{P_1, \ldots, P_s\}\subset Sing(\A)$. Since $P$ is modular, there are lines $\ell'_i\in\A$ through $P$ and $P_i$. Counting the number of lines through $P$ we get $m(P,\A)\geq s+1>s=m(Q,\A)$.

Let $D\in Sing(\A)$ be such that $m(D,\A)=m$. If $D$ is not modular, let $P$ be a modular point of $\A$. Then we have $m(P,\A)>m(D,\A)=m$, contradicting the maximality of $m$.
\end{proof}

\subsection{Dirac-Motzkin conjecture for real supersolvable arrangements}

In \cite{Sy}, Sylvester proposed the following problem: if $\A$ is a full rank real line arrangement in $\mathbb P^2$, then $Sing_2(\A) \neq \emptyset$. In 1944, Gallai solved this problem (see \cite{Ga}), which is now known in the literature as Sylvester-Gallai Theorem.

Dirac and Motzkin conjectured that if $\A$ is a full rank real arrangement of $n$ lines in $\mathbb P^2$, then for $n\geq n_0$ one has $\lvert Sing_2(\A) \rvert \geq n/2$. The existence of the absolute constant $n_0$ is justified in part by the results of Kelly and Moser (\cite[Theorem 3.6]{KeMo}) and Csisma and Sawyer (\cite[Theorem 2.15]{CsSa}) who proved $\lvert Sing_2(\A) \rvert \geq 3n/7$, and $\lvert Sing_2(\A) \rvert \geq 6n/13$ if $n>7$, respectively, as well as by the examples where these two bounds are attained. For $n=7$, the non-Fano arrangement (\cite[Figure 3.1]{KeMo} or \cite[Fig. 3]{CsSa}) has $\lvert Sing_2(\A) \rvert =3$, which clearly fails to be $\geq n/2$; a more complicated example of Crowe and McKee (\cite{CrMc} or \cite[Fig. 4]{CsSa}) is an arrangement of $n=13$ lines with $\lvert Sing_2(\A) \rvert =6$, so that again $6 \ngeq n/2$. Hence we must have $n_0\geq 14$.

The next lemma is crucial to the proof of the main result. The condition that $\mathbb R$ is our base field is necessary for the proof.

\begin{lem}\label{NonmodularLines} Let $\A$ be a full rank supersolvable real line arrangement in $\mathbb P^2$. Let $P\in Sing(\A)$ be a point of max multiplicity (and hence, by Lemma \ref{MaxModular}, a modular point). Then, any line of $\A$ not passing through $P$ has at least one simple singularity.
\end{lem}

\begin{proof}
Define $m = m(P,\A)$. After a linear change of variables, we may assume that $P=[0,0,1]$, so that the lines passing through $P$ are parallel and vertical. Suppose for contradiction that there exists $\ell\in\A$ with $P\notin\ell$, and $\ell\cap Sing_2(\A)=\emptyset$. Since $P$ is modular, we have $\ell \cap Sing(\A) = \{P_1, \ldots, P_m\}$ where each $P_i$ lies at the intersection of $\ell$ and a line through $P$, say $\ell_i'$. We may assume that the $P_i$'s are ordered from left to right (i.e.\ from least homogenized first coordinate to largest).

For each $i$, let $\ell_i \in \A$ be another line through $P$ different from $\ell$ and $\ell_i'$. Note that the $\ell_i$ are distinct, since $\ell$ is the unique line passing through the $P_i$.

Let $a_i$ be the slope of $\ell_i$ for each $i$. Note that each $a_i$ is finite, since $P = [0,0,1]$. We may assume that $a_1 \geq 0$.

Consider the case that all the $a_i$ have the same sign. We may assume that all $a_i > 0$. Let $Q$ be the intersection point of $\ell_1$ and $\ell_m$. If $Q$ lies ``above" $\ell$, then $Q$ necessarily lies to the ``right" of $\ell_m$ and hence, since $a_m \geq 0$, to the ``right" of any of the $\ell_i'$. Similarly, if $Q$ lies ``below" $\ell$, then $Q$ necessarily lies to the ``left" of $\ell_1$ and hence, since $a_1 \geq 0$, to the ``left" of any of the $\ell_i'$. In either case, $Q$ is not on a line through $P$, contradicting that $P$ is modular.

In the remaining case, we may assume that $a_m > 0$. Let $j$ be the least index in which $a_j < 0$, and let $Q$ be the intersection point of $\ell_{j-1}$ and $\ell_j$. Then $Q$ must lie ``between" $\ell_{j-1}'$ and $\ell_j'$, and so cannot lie on a line through $P$. This again contradicts that $P$ is modular.
\end{proof}

\begin{cor}\label{Corollary}
Let $\A$ be a full rank supersolvable real line arrangement in $\Proj^2$. Then $$\lvert Sing_2(\A)\rvert + m(\A) \geq \lvert \A \rvert.$$
\end{cor}
\begin{proof} There are exactly $\lvert \A \rvert-m(\A)$ lines of $\A$ not passing through $P$, and, by Lemma \ref{NonmodularLines}, each of these must have a simple singularity of $\A$ on it. Because $P$ is modular, such a simple point can occur only as the intersection of a line of $\A$ through $P$ and of a line of $\A$ not through $P$. Therefore $\lvert Sing_2(\A) \rvert\geq \lvert \A \rvert-m(\A)$.
\end{proof}

\begin{thm}\label{Main1}
Let $\A$ be a full rank supersolvable real line arrangement in $\mathbb P^2$. Then $$\lvert Sing_2(\A) \rvert \geq\frac{\lvert \A \rvert}{2}.$$
\end{thm}
\begin{proof} Let $n:=\lvert \A \rvert$, and let $m:=m(\A)$. Let $P\in Sing(\A)$ with $m(P,\A)=m$. Then, by Lemma \ref{MaxModular}, $P$ is a modular point.

\medskip

\noindent {\em Case 1.} Suppose $m\leq n/2$. Then $n-m\geq n/2$, and from Corollary \ref{Corollary} we have $\lvert Sing_2(\A) \rvert\geq n/2$.

\medskip

\noindent {\em Case 2.} Suppose $m>n/2$. In general, any $\A\subset\mathbb P^2$ with each line containing a simple singularity has $\lvert Sing_2(\A) \rvert\geq \lvert \A \rvert/2$. Suppose that there is $\ell\in\A$ with no simple singularity on it. From Lemma \ref{NonmodularLines} we necessarily have that $P\in\ell$.

Let $\ell'\in\A$ with $P\notin\ell'$. Let $\{Q\}=\ell\cap\ell'$. Then $m(Q,\A)\geq 3$.

Suppose $\ell'\cap Sing(\A)=\{Q_1,\ldots,Q_{m-1},Q_m\}$, with $Q_m=Q$, and $m(Q_1,\A)=2$ (from Lemma \ref{NonmodularLines}). Denote $n_i:=m(Q_i,\A), i=1,\ldots,m$. It is clear that

\begin{equation} \label{Bezout}
\underbrace{(n_1-1)}_{1}+(n_2-1)+\cdots+(n_m-1)=n-1.\,
\end{equation}

{\em Case 2.1.} Suppose $3n/4\geq m>n/2$. If $n_i\geq 3$, for all $i=2,\ldots,m$, then, from (\ref{Bezout}) above $$n-2\geq 2(m-1)$$ giving $n\geq 2m$; contradiction.d Hence we can assume $n_2=2$ as well, so that $\ell'$ has at least two simple singularities on it. This leads to $$\lvert Sing_2(\A) \rvert\geq 2(n-m)\geq 2n-3n/2=n/2.$$

{\em Case 2.2.} Suppose $5n/6\geq m>3n/4$. From the previous case we can assume $n_1=n_2=2$. Suppose $n_i\geq 3$, for all $i=3,\ldots,m$. Then, from (\ref{Bezout}), $$n-3\geq 2(m-2)$$ giving $(n+1)/2\geq m$. This contradicts $m>3n/4$ and the fact that $n\geq 3$ ($\A$ has rank 3). So we can assume $n_3=2$ as well, which means that $\ell'$ has at least three simple singularities on it. This leads to $$\lvert Sing_2(\A) \rvert\geq 3(n-m)\geq 3n-5n/2=n/2.$$

{\em Case 2.u-1.} Suppose $\frac{2u-1}{2u}\cdot n\geq m>\frac{2u-3}{2u-2}\cdot n$, where $u$ is some integer $>3$. From the inductive hypothesis (since $m>\frac{2u-3}{2u-2}\cdot n>\frac{2(u-1)-3}{2(u-1)-2}\cdot n$) we can assume $n_1=\cdots=n_{u-1}=2$. Suppose $n_i\geq 3$, for all $i=u,\ldots,m$. Then, from (\ref{Bezout}), $$n-u\geq 2(m-u+1)$$ giving $(n+u-2)/2\geq m$.

We have $n-1\geq m$, from the full rank hypothesis. Then $n-1>\frac{2u-3}{2u-2}\cdot n$, which gives $n>2u-2$. The inequality $(n+u-2)/2\geq m$ obtained before, together with $m>\frac{2u-3}{2u-2}\cdot n$, leads to $u-1>n$. But we previously obtained $n > 2u-2$, a contradiction. Hence we can assume $n_u=2$ as well, which means that $\ell'$ has at least $u$ simple singularities on it. This leads to $$\lvert Sing_2(\A) \rvert\geq u(n-m)\geq un-(2u-1)n/2=n/2.$$

\end{proof}

\begin{exm}\label{Example1} For $m\geq 3$, the B\"{o}r\"{o}czky configuration of points is $$X_{2m}:=\underbrace{\{[\cos\frac{2\pi j}{m},\sin\frac{2\pi j}{m},1]:0\leq j<m\}}_{\Lambda_1}$$ $$\cup \underbrace{\{[-\sin\frac{\pi j}{m},\cos\frac{\pi j}{m},0]:0\leq j<m\}}_{\Lambda_2}.$$ Let $\A_{2m}\subset\mathbb P^2$ be the real arrangement of the $2m$ lines dual to the points of $X_{2m}$.

The calculations done in the proof of \cite[Proposition 2.1 (i)]{GrTa} show that $\A_{2m}$ is supersolvable. The dual lines to the points of $\Lambda_2$ all pass through the point $P:=[0,0,1]$, and any two lines dual to two distinct points $[\cos\frac{2\pi j}{m},\sin\frac{2\pi j}{m},1], [\cos\frac{2\pi j'}{m},\sin\frac{2\pi j'}{m},1]\in\Lambda_1$ intersect in a point that belongs to the line with equation $$-\sin\frac{\pi (j+j')}{m}\cdot x+\cos\frac{\pi (j+j')}{m}\cdot y=0.$$ Since $0\leq j,j'\leq m-1$, then $0<j+j'\leq 2m-3$.

If $j+j'\leq m-1$, then the above line is in $\A_{2m}$, and also passes through $P$.

If $j+j'\geq m$, then consider $k:=j+j'-m$, which satisfies $0\leq k\leq m-3<m$. Trig identities $$\sin\frac{\pi (j+j')}{m}=-\sin\frac{\pi k}{m}, \cos\frac{\pi (j+j')}{m}=-\cos\frac{\pi k}{m}$$ give also that the line of equation listed above is in $\A_{2m}$.

Everything put together shows that $P$ is a modular point (of maximum multiplicity $m$), so $\A_{2m}$ is supersolvable and attains the bound of our Theorem \ref{Main1}.
\end{exm}

At large, \cite[Theorem 2.2]{GrTa} shows the uniqueness of the B\"{o}r\"{o}czky examples; the proof is quite challenging. For our specific case of supersolvable arrangements the proof is more intuitive and we believe simpler.

\begin{thm}\label{unique} Let $\A$ be a supersolvable, real line arrangement with $m:=m(\A)\geq 3$ and $\lvert \A \rvert = 2m$. Then, $\lvert Sing_2(\A) \rvert=m$ if and only if $\A$ and $\A_{2m}$ are combinatorially equivalent (i.e., they have isomorphic Orlik-Solomon algebras).
\end{thm}
\begin{proof} We first analyze the geometry of $\A_{2m}$.

Let $L_j:=V(\cos\frac{2\pi j}{m} x+\sin\frac{2\pi j}{m}y+z)$ for $j=0,\ldots,m-1$ and $L'_k:=V(\sin\frac{\pi k}{m} x-\cos\frac{\pi k}{m} y)$ for $k=0,\ldots,m-1$ be the lines of $\A_{2m}$; observe that the modular point $P:=[0,0,1]$ lies on all the $L'_k$'s. We have the following circuits (i.e.\ minimal dependent sets): $$\{L_j,L_{j'},L'_{\overline{j+j'}}\}, j<j'\mbox{ and }\{L'_k,L'_{k'},L'_{k''}\}, k<k'<k'',$$ where $\overline{j+j'}$ denotes the reminder of the division of $j+j'$ by $m$. Note that if $j+j'\equiv j+j'' \; ({\rm mod}\, m)$ with $j',j''\in\{0,\ldots,m-1\}$, then $m \vert (j'-j'')$ and hence $j'=j''$ (since $\lvert j'-j'' \rvert \leq m-1$). Therefore, the intersection points on lines not passing through $P$ are either simple points or triple points.

\medskip

\noindent {\em Claim 1.} There exists exactly one simple point on each $L_j$, namely the intersection of $L_j$ and $L'_{\overline{2j}}$.

Suppose there exists $L_{j'}$ that passes through the same intersection point. Then we must have $\overline{j+j'}=\overline{2j}$, leading to $m \vert (j'-j)$. So $j'=j$.

As a consequence we obtain that if $m$ is even, then half of the lines through $P$ have exactly two simple points and half have no simple points, and if $m$ is odd, then each line through $P$ has exactly one simple point.

\vskip .2in

Let $\A$ be a real supersolvable line arrangement with $m(\A) = m$ consisting of $n=2m$ lines. Suppose that $\lvert Sing_2(\A) \rvert=m$. To show that $\A$ and $\A_{2m}$ are combinatorially equivalent we follow roughly the same ideas as Green and Tao, with the hope that for supersolvable arrangements the argument is more transparent. If for them the Caley-Bacharach Theorem was the key ingredient in the proof, for us a simple plane geometry problem does the trick (see the proof of Claim 2 below).

Let $Q \in Sing(\A)$ with $m(Q,\A)=m$. Then, by Lemma \ref{MaxModular}, $Q$ is a modular point. Let $M'_0,\ldots,M'_{m-1}$ be the lines of $\A$ passing through $Q$. Also, denote by $M_0,\ldots,M_{m-1}$ the remaining $n-m=m$ lines of $\A$.

By Lemma \ref{NonmodularLines}, each line $M_i$ has at least one simple point on it which is the intersection of this $M_i$ and one of the $M'_k$. Since we have $m$ simple points in total and $m$ lines not passing through $Q$, each $M_i$ has exactly one simple point on it. If an $M_i$ has a point with 4 or more lines of $\A$ through it, then, from equation (\ref{Bezout}) in the proof of Theorem \ref{Main1}, we obtain $2(m-1)<n-2$, contradicting that $2m=n$. So all the other points on $M_i$ have multiplicity 3.

Let us consider some arbitrary line $M'\in\A$ through $Q$. Suppose it has $u$ simple points and $v$ triple points. Then, the same equation (\ref{Bezout}) gives $$m-1+u+2v=n-1,$$ leading to $u+2v=n-m=m$.

\medskip

\noindent{\em Claim 2.} $u\leq 2$.

Suppose $u\geq 3$. Pick $M_1,M_2,M_3$ through 3 of the $u$ simple points on $M'$. If $M_1\cap M_2\cap M_3=\{Q'\}$, then since $Q$ is modular, then there is the line connecting $Q'$ and $Q$ that leads to $m(Q',\A)\geq 4$. Contradiction. So $M_1\cap M_2 = \{Q_{1,2}\}, M_1\cap M_3 = \{Q_{1,3}\}, M_2\cap M_3 = \{Q_{2,3}\}$, with these three points distinct. Since $Q$ is modular, it must connect to these three points through some lines $M'_3,M'_2,M'_1$ respectively. Intersecting $M_i$ and $M'_i$, for $i=1,2,3$, we obtain $Q_i$. The lines $M_1,M_2,M_3$ already have their simple points, so $m(Q_i,\A)=3, i=1,2,3$. Also, $Q$ is already connected to each $Q_i$, so the extra line through each $Q_i$ that adds up to the total multiplicity of 3, must come from a line not through $Q$ and different than $M_1,M_2,M_3$.

The geometry of the real plane cannot allow for the points $Q_i$ to be collinear (see figure below): we have the triangle $\triangle(Q_{1,2}Q_{1,3}Q_{2,3})$ and a point $Q$ not on any of its edges $M_1, M_2,M_3$. Let $Q_1=QQ_{2,3}\cap M_1, Q_2=QQ_{1,3}\cap M_2,Q_3=QQ_{1,2}\cap M_3$. Then $Q_1,Q_2,Q_3$ are not collinear.

\begin{center}
\epsfig{file=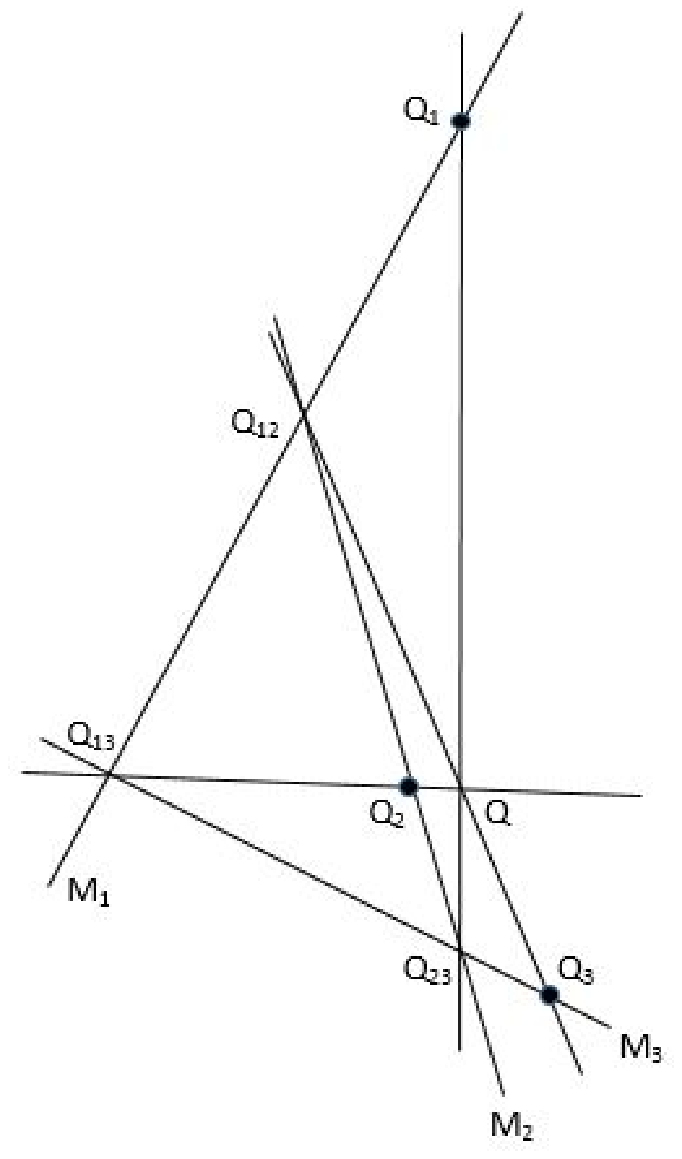,height=3.5in,width=3in}
\end{center}

Suppose one such extra line contains $Q_1$ and $Q_2$. Then it intersects $M_3$ in a ``new'' point $Q'_3$. The extra line passing through $Q_3$ should not contain either $Q_1$, nor $Q_2$ because it will make their multiplicity bump to 4. So this extra line passing through $Q_3$ intersects the lines $M_1$ and $M_2$ in two ``new'' points $Q'_1$ and $Q'_2$. In the case when each $Q_i$ comes with its own extra line, we still obtain three ``new'' points $Q'_1$, $Q'_2$, $Q'_3$ on $M_1,M_2,M_3$, respectively. Now with these three ``new'' non-collinear (by the same geometry argument as above) points replacing the three points $Q_{i,j}$, we can continue the argument repeatedly until we exhaust the lines through $Q$, obtaining a contradiction.

\medskip

Claim 2 tells us that if $m$ is odd, then $u=1$, and hence each line through $Q$ has exactly one simple point, and if $m$ is even, then $u=0$ or $u=2$, and hence, keeping in mind that there are exactly $m$ simple points, half of the lines through $Q$ have exactly two simple points and half have no simple point.

\medskip

To conclude the proof, we show that the Orlik-Solomon algebras are isomorphic via an isomorphism of NBC-bases (see \cite{OrTe} for definitions and more information, but more specifically \cite{Pe}). The abbreviation NBC stands for {\em non broken circuit} (in \cite[Definition 2.1]{Pe} it is called {\em basic}), and it means the following: Suppose one picks an ordering of the hyperplanes of an arrangement $\A = \{H_1, \ldots, H_n\}$. A broken ciruit (under this chosen ordering) is a subset $S \subseteq A$ such that there is $H \in \A$ with $H$ smaller (with respect to the ordering) than all elements of $S$ and such that $S \cup \{H\}$ is a circuit. An NBC is an independent subset of $\A$ that does not contain a broken circuit.

$\A_{2m}$ is supersolvable, and suppose we order the lines of this arrangement as $$L'_0<L'_1<\cdots<L'_{m-1}<L_0<L_1<\cdots<L_{m-1}.$$ Then \cite[Theorem 3.22]{Pe} (or the original \cite[Theorem 2.8]{BjZi}) says that the Orlik-Solomon algebra, $OS(\A_{2m})$, is quadratic. The quadratic NBC elements in its basis consisting of the following: $$\{(L'_0,L'_i): 1\leq i\leq m-1\}\cup\{(L_j,L'_{\overline{2j}}):0\leq j\leq m-1\}.$$

Similarly, $\A$ is supersolvable, and if we order its lines as $$M'_0<M'_1<\cdots<M'_{m-1}<M_0<M_1<\cdots<M_{m-1}$$ then $OS(\A)$ is quadratic, with the quadratic NBC elements in its basis consisting of the following: $$\{(M'_0,M'_i): 1\leq i\leq m-1\}\cup\{(M_j,M'_{\delta(j)}):0\leq j\leq m-1\},$$ where $M'_{\delta(j)}$ is the line passing through $Q$ such that $M_j\cap M'_{\delta(j)}$ is the unique simple point on $M_j$.

From Claim 1, it is clear that the two Orlik-Solomon algebras are isomorphic.
\end{proof}

\subsection{The Slope problem.} Given $n\geq 3$ points in the (real) plane, not all collinear, they determine at least $n-1$ lines distinct slopes. The bound can be achieved only if $n\geq 5$ is odd. In \cite{JaHi} 4 infinite families and 102 sporadic examples are shown to satisfy the equality in the bound (here we cited from \cite[Chapter 11]{AiZi}).

The connection with supersolvable line arrangements is the following: Consider $P_1,\ldots,P_n\in\mathbb R^2, n\geq 3$ not all collinear. We consider these points in $\mathbb P(\mathbb R^2)$, by adding one extra (homogeneous) coordinate equal to 1. Two lines in $\mathbb R^2$ are parallel (respectively, identical) and have the same slope, if, when embedded in $\mathbb P(\mathbb R^2)$, they intersect one another at the line at infinity (of equation $z=0$).

Next let $D_1,\ldots, D_w$ to be all the points lying at the line of infinity, obtained by intersecting this line at infinity with all the lines determined by $P_1,\ldots,P_n$. Then, $w$ is the number of distinct slopes of all these lines. The Slope problem says that $$w\geq n-1.$$

Dualizing the above construction, we obtain the following:

$\bullet$ $P_1,\ldots,P_n$ become the lines $\ell_1,\ldots,\ell_n\in\mathbb P^2$, and $D_1,\ldots,D_w$ become the lines $\delta_1,\ldots,\delta_w\in \mathbb P^2$.

$\bullet$ The line at infinity (of equation $z=0$) becomes the point $[0,0,1]\in\mathbb P^2$, which belongs to every $\delta_i,i=1,\ldots,w$.

$\bullet$ $P_i$ and $P_j$ determine the line $\ell$ if and only if $\ell_i$ and $\ell_j$ intersect in the point dual to $\ell$. Furthermore, since $\ell$ intersects the line at infinity at some $D_k$, the point dual to $\ell$ belongs to the line $\delta_k$.

Consider the line arrangement $\A_{PD}=\{\ell_1,\ldots,\ell_n,\delta_1,\ldots,\delta_w\}\subset\mathbb P^2$. Then, the three bullets above (especially the last two) imply that $\A_{PD}$ is supersolvable (with a modular point $P_{mod}:=[0,0,1]$).

We have $m(P_{mod})=w$ and $n=\lvert \A_{PD}\rvert-w$, therefore $$w\geq n-1\mbox{ if and only if } m(P_{mod})\geq\frac{\lvert \A_{PD}\rvert-1}{2}.$$ This statement leaded us to the following

\begin{prop}\label{equivalence} The following are equivalent:
\begin{enumerate}
  \item[(1)] The Slope problem.
  \item[(2)] If $\A$ is a full rank real supersolvable line arrangement in $\mathbb P^2$, then $$m(\A)\geq\frac{|\A|-1}{2}.$$
\end{enumerate}
\end{prop}
\begin{proof} First let us note that if a supersolvable line arrangement $\A$ has two modular points $Q_1$ and $Q_2$, of distinct multiplicities $m(Q_1)>m(Q_2)$, then $m(Q_1)+m(Q_2)=\lvert\A\rvert+1$, meaning that any line of $\A$ must pass through $Q_1$ or $Q_2$. If there were a line $\ell$ not passing through $Q_1$ and $Q_2$, then from the modularity of these two points, the number of points on $\ell$ is precisely $m(Q_1)$. However, it must also be $m(Q_2)$, giving that $m(Q_1) = m(Q_2)$, a contradiction.

If we are in the situation of the observation above, then $m(\A)\geq m(Q_1)>(\lvert\A\rvert+1)/2$, hence (2) in the statement is satisfied immediately. Also, if the line arrangement $\A_{PD}$ has $m(P_{mod})<m(\A_{PD})$, then $\ell_1,\ldots,\ell_n$ all pass through the modular point of $\A_{PD}$ of maximum multiplicity $m(\A_{PD})$, and therefore the points $P_1,\ldots,P_n$ we started with are collinear, contradicting the conditions of (1). So $m(P_{mod})=m(\A_{PD})$.

With the two observations above, the equivalence is evident. That (2) implies (1) is immediate, and for (1) implies (2) we may pick a (modular) point of maximum multiplicity and change variables so that this point is $P_{mod}=[0,0,1]$. Then the line arrangement becomes an arrangement of the form $\A_{PD}$. \end{proof}

\section{Supersolvable line arrangements over any field of characteristic 0}

\subsection{The degree of the reduced Jacobian scheme.} In this section we assume $\mathbb K$ to be any field of characteristic 0. Let $\A\subset\mathbb P_{\mathbb K}^2$ be a full rank supersolvable arrangement of $n$ lines. Let $m:=m(\A)$ be the multiplicity of $\A$, and suppose $m\geq 3$. Let $P\in Sing(\A)$ with $m(P,\A)=m$; therefore by Lemma \ref{MaxModular}, $P$ is modular.

In $R:=\mathbb K[x,y,z]$, let $f$ be the defining polynomial of $\A$, let $g$ be the product of linear forms defining the $m$ lines of $\A$ passing through $P$, and let $h$ be the product of the linear forms defining the $n-m$ lines of $\A$ not passing through $P$. Without loss of generality we may assume $$f=g\cdot h.$$

Let $\A_h:=V(h)\subset\mathbb P^2$; this is the line arrangement consisting of all the lines of $\A$ not passing through $P$. The following equivalent statement is immediate: $Q\in Sing(\A_h)$ if and only if $Q\in Sing(\A)$ and $m(Q,\A)\geq 3$. Therefore, taking into account that $P$ is not a simple point of $\A$ ($m\geq 3$), we have the formula

\begin{equation} \label{SingularityEQ}
\lvert Sing_2(\A) \rvert=\lvert Sing(\A) \rvert - \lvert Sing(\A_h) \rvert-1.\,
\end{equation}

This is the main reason why the study of $\lvert Sing(\A) \rvert$ goes together with the study of $\lvert Sing_2(\A) \rvert$. In fact we have the following immediate result:

\begin{prop}\label{bound} Let $\A$ be a full rank supersolvable arrangement of $n$ lines, with max multiplicity $m\geq 3$. Then $$\lvert Sing_2(\A) \rvert \geq 2\lvert Sing(\A) \rvert-m(n-m)-2.$$ Equality holds if and only if $\A_h$ is generic (i.e., $Sing_2(\A_h)=Sing(\A_h)$), and in this case $\lvert Sing_2(\A) \rvert=(n-m)(2m-n+1)$.
\end{prop}
\begin{proof} Let $P\in Sing(\A)$ be such that $m(P,\A)=m$. Then $P$ is modular. Let $\ell\in\A$ be an arbitrary line passing through $P$. Suppose there are $s$ simple points on $\ell$, and $t$ multiple points on $\ell$, distinct from $P$, with multiplicities $n_1,\ldots,n_t\geq 3$. Equation (\ref{Bezout}) in the proof of Theorem \ref{Main1} gives: $$(m-1)+s+(n_1-1)+\cdots+(n_t-1)=n-1.$$ As $n_i\geq 3$ leads to $$n-m-s\geq 2t.$$ Summing over all the lines through $P$ implies $$m(n-m)-\lvert Sing_2(\A) \rvert\geq 2 \lvert Sing(\A_h) \rvert.$$

Formula (\ref{SingularityEQ}) above proves the assertion, and the ``if and only if'' statement. If $\A_h$ is generic, then $\lvert Sing(\A_h)\rvert={{n-m}\choose{2}}$, and we immediately obtain the claimed formula.
\end{proof}

It is worth noting that for the arrangement(s) in Theorem \ref{unique}, $\lvert Sing(\A)\rvert={{m+1}\choose{2}}+1$. To see this calculate $u+v$ for each line through $P$, sum these and add 1 to account for $P$. Then formula (\ref{SingularityEQ}), together with $\lvert Sing_2(\A)\rvert=m$, gives $$\lvert Sing(\A_h)\rvert={{m}\choose{2}},$$ meaning that $\A_h$ is generic (as $\lvert\A_h\rvert=n-m=m$).

\subsubsection{An interesting example.} For any homogeneous polynomial $\Delta\in R$, let $J_{\Delta}=\langle \Delta_x,\Delta_y,\Delta_z\rangle\subset R$ be the {\em Jacobian ideal of $\Delta$}. Since $\Delta$ is a homogeneous polynomial, $J_{\Delta}$ defines the scheme of the singular locus of the divisor $V(\Delta)\subset \mathbb P^2$. A singular point shows up in $J_{\Delta}$ with a certain multiplicity (known in the literature as the Tjurina number), but we are interested only in the number of singular points, which is the degree of $\sqrt{J_{\Delta}}$ (the defining ideal of the reduced Jacobian scheme).

The homological information of $J_f$, where $f$ is the defining polynomial of a supersolvable line arrangement, is very well understood. For example, the first syzygies module of $J_f\subset R$ is a free $R$-module of rank 2, with basis elements having degrees $m-1$ and $n-m$ (see \cite{JaTe}, or \cite{OrTe}).

By \cite[Theorem 2.2]{To2}, there exists $(\alpha,\beta,\gamma)$ a syzygy on $J_f$ with $\{\alpha,\beta,\gamma\}$ forming a regular sequence. The degree of this syzygy is $d$ (equal to $m-1$ or $n-m$). Then, if $m<n-1$, by \cite[Proposition 3.1]{To2}, $\lvert Sing(\A) \rvert\leq d^2+d+1$.

When \cite{To2} appeared, there was a question as to if the bound can be attained. The next example shows that it can. We came across this example while trying to prove Proposition \ref{equivalence}(2) without using the Slope problem. We have so far been unable to find such a proof.

Consider the line arrangement $\A$ with defining polynomial $f=xyz(x-y)(x-z)(y-z)(x+y-z)(x-y+z)(x-y-z)$.

\begin{center}
\epsfig{file=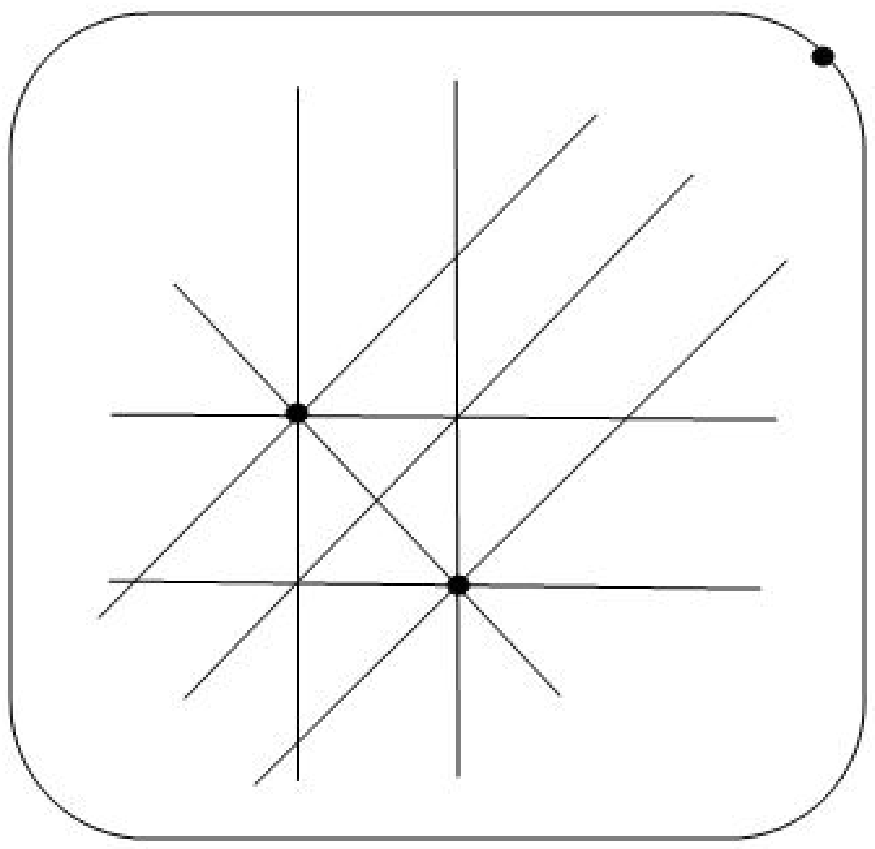,height=2.3in,width=2.3in}
\end{center}

We have $n=9$ and $m=4$. This is the projective picture with the border being the line at infinity $z=0$. The marked points are modular points of maximum multiplicity.

The calculations below were done with \cite{GrSt}.

$P=[1,1,0]\in Sing(\A)$ with $m(P,\A)=4$; the lines of equation $z=0,x-y=0,x-y+z=0,x-y-z=0$ all pass through $P$. In fact, since $z(x-y)(x-y+z)(x-y-z)\in\sqrt{J_{f}}$, by \cite[Theorem 2.2]{To1}, $P$ is a modular point and $\A$ is supersolvable.

Calculations show that

\begin{eqnarray}
&&(\frac{1}{10}x^2z-\frac{11}{30}xyz+\frac{1}{10}y^2z+\frac{1}{12}xz^2+\frac{1}{12}yz^2-\frac{1}{20}z^3,\nonumber\\ &&\frac{1}{5}x^2y+\frac{1}{6}xy^2-\frac{1}{10}y^3-\frac{11}{15}xyz+\frac{1}{6}y^2z+\frac{1}{5}yz^2,\nonumber\\ &&-\frac{3}{10}x^3+\frac{1}{2}x^2y+\frac{3}{5}xy^2+\frac{1}{2}x^2z-\frac{11}{5}xyz+\frac{3}{5}xz^2) \nonumber
\end{eqnarray} is a syzygy on $J_{f}$ (i.e., on some linear combination of the partial derivatives of $f$). Its entries generate an ideal of height 3, and hence they form a regular sequence. The degree of this syzygy is $3=m-1$.

By \cite[Proposition 3.1]{To2}, $\lvert Sing(\A) \rvert\leq 3^2+3+1=13$, and calculations show that in fact we have equality.

\subsection{Dirac-Motzkin conjecture for supersolvable arrangements over $\mathbb C$.} \cite[Theorem 2]{Di} shows that any complex smooth cubic plane curve has exactly nine inflection points, and Theorem 3 in the same paper proves that any two of these points are collinear with a third (an image of such a curve can be found by searching for ``Hesse configuration''). This shows that Sylvester's original problem does not have a solution over $\mathbb C$.

Such a configuration of points (where any two points in the set are collinear with a third in the set) is called {\em Sylvester-Gallai configuration (SGC)}. Serre proposed the following problem: Any SCG in $\mathbb C^n$ with $n \geq 3$ is contained in a (2-dimensional) plane. The problem was solved by Kelly, for $n=3$ (see \cite[Theorem]{Ke}), and for any $n\geq 3$, by Elkies-Pretorius-Swanepoel (see \cite[Theorem 2]{ElPrSw}; they reduced the problem to the case of $n=3$, though they use a different method than Kelly).

In \cite[Example 4.1]{Zi} it is shown that the line arrangement dual to these nine points (called the Hessian arrangement) is a free arrangement with exponents $(1,4,4)$, but it is not supersolvable. Alternatively, we can see this by using the formula in Proposition \ref{bound}. For, we have $n=9, m=3$ and $\lvert Sing_2(\A)\rvert=0$; if $\A$ were supersolvable, the formula would give $\lvert Sing(\A)\rvert\leq 10$. However, $\A$ has at least 12 singularities corresponding to the 12 lines determined by the 9 inflection points.

\medskip

We conjecture the following

\begin{conj}\label{conjecture} Let $\A\subset\mathbb P^2$ be a full rank complex supersolvable arrangement of $n$ lines. Then $$\lvert Sing_2(\A)\rvert\geq \frac{n}{2}.$$
\end{conj}

Theorem \ref{Main1} relies on the arrangement being real only in the use of Lemma \ref{NonmodularLines}; everything else is independent of the base field and relies on the observation we made at the end of Introduction. The key in proving Conjecture \ref{conjecture} lies within proving a version of Lemma \ref{NonmodularLines} valid over $\mathbb{C}$.

Suppose $\A\subset\mathbb P^2$ is a full rank supersolvable complex arrangement of $n$ lines. Let $P\in Sing(\A)$ be a point of maximum multiplicity $m:=m(\A)$. Suppose (for contradiction) that there is a line $P\notin\ell\in\A$ with no simple points on it. If there is $Q_i\in\ell\cap Sing(\A)$, with $m(Q_i,\A)\geq 4$, by removing any line through $Q_i$ not passing through $P$ and distinct from $\ell$, the resulting complex arrangement is still supersolvable, has $n-1$ lines, and still has the line $\ell$ with no simple points on it. By this observation, the following conjectured result will prove by contradiction the complex version of Lemma \ref{NonmodularLines}.

\begin{conj}\label{conj2} Let $\A\subset\mathbb P^2$ be a full rank supersolvable arrangement of $n$ lines and with singular modular point $P$ of maximum multiplicity $m\geq 3$. If there exists a line $\ell\in\A$ not passing through $P$ and only with triple singularities on it, then $\A$ is not realizable over $\mathbb C$.
\end{conj}

Because $\ell$ does not pass through $P$, it must have exactly $m$ singularities on it: $P_1,\ldots,P_m$. The formula $\sum(m_{P_i}-1)=n-1$, together with $m_{P_i}=3$, gives that $n=2m+1$.

After a change of variables, we may assume that $P=[0,0,1]$ and $\ell=V(z)$.

Suppose $\ell_i=V(a_ix+b_iy), i=1,\ldots,m$ are the lines of $\A$ through $P$, with $a_ib_j-a_jb_i\neq 0$, for all $i\neq j$. Suppose each point $P_i$ is the intersection of $\ell_i$, $\ell$, and another line $\ell_i'$ of $\A$ not passing through $P$. Then we can assume $\ell_i'=V(a_ix+b_iy+z), i=1,\ldots,m$.

The supersolvable condition translates to the following: for $i\neq j$, the point of intersection of $\ell_i'$ and $\ell_j'$ must lie on a line through $P$. This is equivalent to the following determinantal condition: for any $1\leq i<j\leq m$, there exists $k(i,j)\in\{1,\ldots,m\}-\{i,j\}$ such that
$$\left|\begin{array}{lll}
a_i&b_i&1\\
a_j&b_j&1\\
a_{k(i,j)}&b_{k(i,j)}&0
\end{array}\right|=0.$$

\begin{exm}\label{exampleB} Let $\mathcal B(m), m\geq 3$ be the supersolvable line arrangement with the following additional dependencies: $(\ell_1',\ell_2',\ell_3), (\ell_1',\ell_3',\ell_4),\ldots,(\ell_1',\ell_{m-1}',\ell_m)$, $(\ell_1',\ell_m',\ell_2)$, and $(\ell_2',\ell_3',\ldots,\ell_m',\ell_1)$ (we use this notation to indicate that these lines are all concurrent). The picture for the matroid of $\mathcal{B}(5)$ is shown below. The line arrangement $\mathcal B(m)$ satisfies the conditions of Conjecture \ref{conj2}. We show that $\mathcal B(m)$ is not realizable over $\mathbb C$.

\begin{center}
\epsfig{file=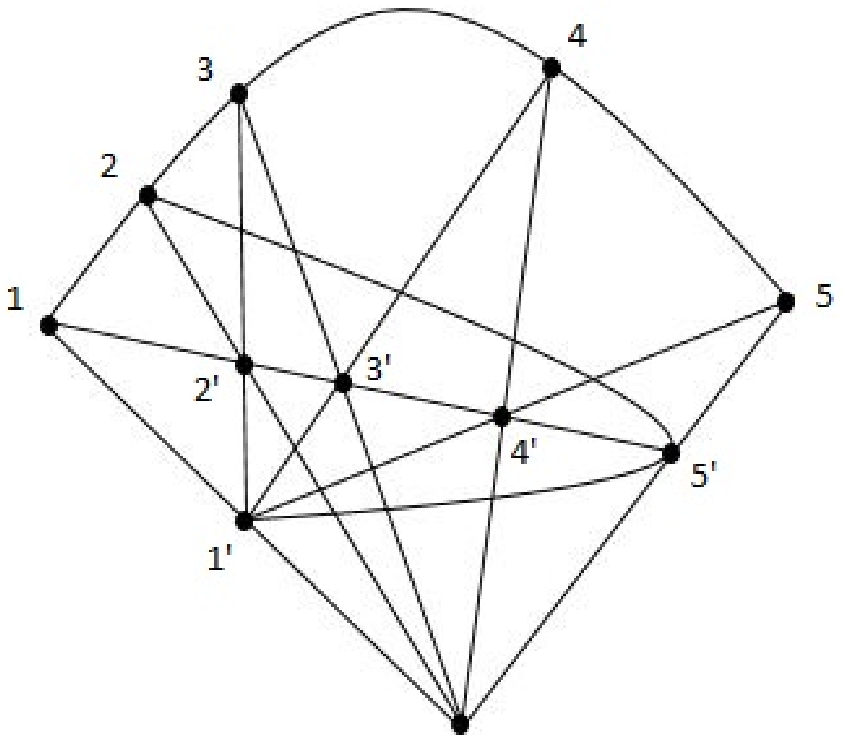,height=2.3in,width=2.3in}
\end{center}

\noindent {\bf Claim.} For any $i=0,\ldots,m-2$, we have $$a_{m-i}b_2-a_2b_{m-i}=(i+1)(a_1b_2-a_2b_1).$$

\noindent {\bf Proof of Claim.} We use induction on $i\geq 0$.

The dependency $(\ell_1',\ell_m',\ell_2)$ gives $a_mb_2-a_2b_m=a_1b_2-a_2b_1$. So the base case of induction, $i=0$, is shown.

Suppose $i>0$ and suppose $a_{m-i+1}b_2-a_2b_{m-i+1}=i(a_1b_2-a_2b_1)$. The dependency $(\ell_1',\ell_{m-i}',\ell_{m-i+1})$ gives $$a_{m-i}b_{m-i+1}-a_{m-i+1}b_{m-i}=a_1b_{m-i+1}-a_{m-i+1}b_1.$$ The dependency $(\ell_2',\ell_{m-i+1}',\ell_1)$ gives $$a_1b_{m-i+1}-a_{m-i+1}b_1 = a_1b_2-a_2b_1,$$ and the dependency $(\ell_2',\ell_{m-i}',\ell_{m-i+1}')$ gives $$a_{m-i}b_2-a_2b_{m-i}=(a_{m-i+1}b_2-a_2b_{m-i+1})+(a_{m-i}b_{m-i+1}-a_{m-i+1}b_{m-i}).$$ The induction hypotheses proves the claim.

\medskip

Making $i=m-2$ in the Claim, we obtain $$(m-1)(a_1b_2-a_2b_1)=0.$$ Since we work over $\mathbb C$, we obtain the contradiction $a_1b_2-a_2b_1=0$.
\end{exm}

Next we look at arrangements satisfying Conjecture \ref{conj2}, for $m=3,4,5$. It turns out that either these arrangements are isomorphic (from the matroid point of view) to $\mathcal B(m)$ or have a subarrangement isomorphic to $\mathcal B(m-1)$ or $\mathcal B(m-2)$. Hence they are not realizable over $\mathbb C$. We are tempted to ask if this happens for any value of $m$, because if that is the case Example \ref{exampleB} will prove Conjecture \ref{conj2}.

\subsection{Case $m=3$.} Then we must have the dependencies $(\ell_1',\ell_2',\ell_3), (\ell_1',\ell_3',\ell_2)$, and $(\ell_2',\ell_3',\ell_1)$. But this is exactly $\mathcal B(3)$. In fact, the matroid of this line arrangement is the Fano projective plane, and the proof of Example \ref{exampleB} has the same conclusion as \cite[Proposition 6.4.8(i)]{Ox}.

\subsection{Case $m=4$.} In general, if we do not want to reduce to the case $m-1$, we should impose also the condition that for every $\ell_i$ there are exist two $\ell_{\mu(i)}'$ and $\ell_{\delta(i)}'$ such that they all intersect at a point. With this in mind, and since we do not want to have $\mathcal B(3)$ as a subarrangement, for $m=4$, after some relabeling we have the following dependencies: $(\ell_1',\ell_2',\ell_3)$, $(\ell_2',\ell_3',\ell_1)$, and $(\ell_1',\ell_3',\ell_4)$. Suppose $\ell_2',\ell_3',\ell_4',\ell_1$ are not concurrent. Then we must have the dependency $(\ell_2',\ell_4',\ell_3)$, and hence all of these lines are concurrent with $\ell_1'$. So after some relabeling, we may assume the dependency $(\ell_2',\ell_3',\ell_4',\ell_1)$. Then we are left with the only possibility for the dependency $(\ell_1',\ell_4',\ell_2)$. This is exactly $\mathcal B(4)$.

\subsection{Case $m=5$.} In what follows, for the sake of brevity, we leave
out the words ``concurrent'' or ``dependency,'' instead using the brackets
notation introduced previously; as usual, $\neg$ stands for ``not.''

After relabeling we have $(\ell_1',\ell_2',\ell_3),(\ell_1',\ell_3',\ell_4),(\ell_2',\ell_3',\ell_1)$.

$\bullet$ Suppose $(\ell_2',\ell_3',\ell_4',\ell_1)$ (meaning these four lines are concurrent). Then, since we want to avoid constructing a $\mathcal B(4)$, we must have $(\ell_1',\ell_4',\ell_5)$.

If $(\ell_2',\ell_3',\ell_4',\ell_5',\ell_1)$, then $(\ell_1',\ell_5',\ell_2)$ is a must (see beginning of case $m=4$). Then we just constructed a $\mathcal B(5)$.

If $\neg(\ell_2',\ell_3',\ell_4',\ell_5',\ell_1)$, then $(\ell_2',\ell_5',\ell_i)$ with $i\neq 1,2,5$. We have the following subcases.

(A) Suppose $(\ell_2',\ell_5',\ell_3)$. Then, $(\ell_1',\ell_2',\ell_5',\ell_3)$. Then also $(\ell_3',\ell_5',\ell_4)$ or $(\ell_3',\ell_5',\ell_2)$.

(A.1) In the first case, we must have $(\ell_1',\ell_3',\ell_5',\ell_4)$ which contradicts $(\ell_1',\ell_2',\ell_5',\ell_3)$.

(A.2) In the second case, we must have $(\ell_4',\ell_5',\ell_i)$, with $i\neq 1,4,5$. If $(\ell_4',\ell_5',\ell_2)$, then $(\ell_3',\ell_4',\ell_5',\ell_2)$, contradicting $(\ell_3',\ell_4',\ell_1)$; and if $(\ell_4',\ell_5',\ell_3)$, then $(\ell_1',\ell_2',\ell_4',\ell_5',\ell_3)$, contradicting $(\ell_2',\ell_4',\ell_1)$.

(B) Suppose $(\ell_2',\ell_5',\ell_4)$. Then, $(\ell_1',\ell_5',\ell_i)$, with $i\neq 1,5$.

(B.1) If $(\ell_1',\ell_5',\ell_3)$, then $(\ell_1',\ell_2',\ell_5',\ell_3)$; contradiction with $(\ell_2',\ell_5',\ell_4)$.

(B.2) If $(\ell_1',\ell_5',\ell_4)$, then $(\ell_1',\ell_3',\ell_5',\ell_4)$, and then $(\ell_1',\ell_2',\ell_3',\ell_5',\ell_4)$; contradiction with $(\ell_2',\ell_3',\ell_1)$.

(B.3) If $(\ell_1',\ell_5',\ell_2)$, then $(\ell_3',\ell_5',\ell_j)$, with $j\neq 1,3,5$. If $(\ell_3',\ell_5',\ell_2)$, then $(\ell_1',\ell_3',\ell_5',\ell_2)$; contradiction with $(\ell_1',\ell_3',\ell_4)$. Otherwise, if $(\ell_3',\ell_5',\ell_4)$, then $(\ell_2',\ell_3',\ell_5',\ell_4)$; contradiction with $(\ell_2',\ell_3',\ell_1)$.

$\bullet$ Suppose $\neg(\ell_2',\ell_3',\ell_4',\ell_1)$. Then $(\ell_2',\ell_4',\ell_i)$, with $i\neq 1,2,4$. If $(\ell_2',\ell_4',\ell_3)$, then $(\ell_1',\ell_2',\ell_4',\ell_3)$. By relabeling $1\rightarrow 3, 2\rightarrow 2, 3\rightarrow 1,4\rightarrow 4,5\rightarrow 5$, we are in the situation of the previous bullet.

Suppose $(\ell_2',\ell_4',\ell_5)$. We must have $(\ell_3',\ell_4',\ell_i)$, with $i\neq 1,3,4$. We have the following cases.

(A') If $(\ell_3',\ell_4',\ell_5)$, then $(\ell_2',\ell_3',\ell_4',\ell_5)$; contradiction with $(\ell_2',\ell_3',\ell_1)$.

(B') Suppose $(\ell_3',\ell_4',\ell_2)$. We must have $(\ell_1',\ell_4',\ell_j)$, with $j\neq 1,4$.

(B'.1) If $(\ell_1',\ell_4',\ell_2)$, then $(\ell_1',\ell_3',\ell_4',\ell_2)$; contradiction with $(\ell_1',\ell_3',\ell_4)$.

(B'.2) If $(\ell_1',\ell_4',\ell_5)$, then $(\ell_1',\ell_2',\ell_4',\ell_5)$; contradiction with $(\ell_1',\ell_2',\ell_3)$.

(B'.3) If $(\ell_1',\ell_4',\ell_3)$, then $(\ell_1',\ell_2',\ell_4',\ell_3)$; contradiction with $(\ell_2',\ell_4',\ell_5)$.

\medskip

Everything put together gives that for $m=5$, the arrangement is isomorphic with $\mathcal B(5)$, or it has a $\mathcal B(4)$, or a $\mathcal B(3)$ subarrangement. In conclusion, the line arrangement is not realizable over $\mathbb C$.

\begin{cor} Conjecture \ref{conjecture} is true for any full rank supersolvable complex arrangement of $n$ lines, with $n\leq 12$.
\end{cor}
\begin{proof} Suppose there exists $\ell\in \A$, not passing through the modular point of max multiplicity $m$, that has no simple singularities. Suppose $P_1,\ldots,P_m$ are the singularities on $\ell$. Then $$2m\leq\sum_{i=1}^m(m_{P_i}-1)=n-1\leq 11,$$ gives $m\leq 5$. With what we discussed above, such an arrangement is not realizable over $\mathbb C$.
\end{proof}

\medskip

Conjecture \ref{conj2} is equivalent to the following conjecture with a commutative algebraic flavor. This is justified by the fact that the point of intersection of $\ell_i'$ and $\ell_j'$, $i\neq j$ must lie on a line $\ell_{k(i,j)}$.

\begin{conj}\label{conj3} Let $f_1,\ldots,f_m\in\mathbb C[x,y]$ be any $m$ linear forms with $\gcd(f_i,f_j)=1$, for all $i\neq j$. Then $$\prod_{i=1}^m f_i\notin \bigcap_{1\leq i<j\leq m}\langle f_i+z,f_j+z\rangle,$$ where $\langle f_i+z,f_j+z\rangle$ is the ideal of $\mathbb C[x,y,z]$ generated by $f_i+z$ and $f_j+z$.
\end{conj}

\vskip .1in

\noindent{\bf Acknowledgement} We are grateful to Prof. G\"{u}nter Ziegler for many enlightening discussions regarding the complex version of Dirac-Motzkin.

\bibliographystyle{amsalpha}

\end{document}